\documentclass{amsart}
\usepackage[margin=1.22in]{geometry}
\theoremstyle{plain}
\usepackage{hyperref}

\usepackage{amssymb,amsmath,amsfonts,amsthm,epsfig,amscd,stmaryrd}
\usepackage{stmaryrd}

\newtheorem{theorem}[subsection]{Theorem}
\newtheorem{lemma}[subsection]{Lemma}
\newtheorem{example}[subsection]{Example}

\newcommand{\C}{\mathbf{C}}
\newcommand{\F}{\mathbf{F}}
\newcommand{\Q}{\mathbf{Q}}
\newcommand{\R}{\mathbf{R}}
\newcommand{\Z}{\mathbf{Z}}
\newcommand{\Fbar}{\overline{\F}}

\newcommand\im{\mathrm{im}}
\newcommand\rhobar{\overline{\rho}}
\newcommand\ab{\mathrm{ab}}
\newcommand\et{\mathrm{et}}
\newcommand\Frob{\mathrm{Frob}}

\DeclareMathOperator\Res{\mathrm{Res}}
\DeclareMathOperator{\GL}{GL}
\DeclareMathOperator{\GSp}{GSp}
\DeclareMathOperator{\Sp}{Sp}
\DeclareMathOperator{\PSp}{PSp}
\DeclareMathOperator{\PGSp}{PGSp}
\DeclareMathOperator{\End}{End}
\DeclareMathOperator\Jac{\mathrm{Jac}}
\DeclareMathOperator\Pic{\mathrm{Pic}}
\DeclareMathOperator\SU{\mathrm{SU}}
\DeclareMathOperator\USp{\mathrm{USp}}

\title{Some Modular Abelian Surfaces}

\subjclass[2010]{11G10, 11F46; 11Y40, 11F80}

\author[F. Calegari]{Frank Calegari}  \email{fcale@math.uchicago.edu} \address{The University of Chicago,
5734 S University Ave,
Chicago, IL 60637, USA}
\address{School of Mathematics and Statistics, University of Melbourne, Parkville VIC 3010, Australia}

\author[S. Chidambaram]{Shiva Chidambaram}  \email{shivac@uchicago.edu} \address{The University of Chicago,
5734 S University Ave,
Chicago, IL 60637, USA}

\author[A. Ghitza]{Alexandru Ghitza} \email{aghitza@alum.mit.edu}
\address{School of Mathematics and Statistics, University of Melbourne, Parkville VIC 3010, Australia}

\thanks{The first author was supported in part by NSF Grant
  DMS-1701703. }

\begin{document}

\begin{abstract}
In this note, we use the main theorem of~\cite{BCGP} to give explicit examples of modular abelian surfaces~$A$ with~$\End_{\C} A = \Z$
and~$A$ smooth outside~$2$, $3$, $5$, and~$7$.
\end{abstract}

\maketitle
\setcounter{tocdepth}{2}
{\footnotesize
}

\section{Introduction}
Let~$C/\Q$ be a smooth projective curve of genus~$g$. Let\footnote{There is some ambiguity in the literature
as to whether one defines~$\Gamma_{\C}(s)$ to be~$ (2 \pi)^{-s} \Gamma(s)$  or~$\Gamma_{\R}(s) \Gamma_{\R}(s+1)  = 2 \cdot (2 \pi)^{-s} \Gamma(s)$. It makes no difference
as long as one uses the same choice for both~$\Lambda(C,s)$ and~$\Lambda(\pi,s)$.
To be safe, we make the same choice as Serre~\cite[\S3(20)]{Serre}.}~$\Gamma_{\C}(s) = (2 \pi)^{-s} \Gamma(s)$.
Associated to~$C$ and its Jacobian~$A = \Jac(C)$
is a completed~$L$-function
$$\Lambda(C,s) = 
\Gamma_{\C}(s)^g
 \prod_{p} L_p(C,p^{-s})^{-1},$$
where, for any~$\ell \ne p$, $L_p(C,T) = \det \left(I_{2g} - T \cdot \Frob_p\mid H^1_{\et}(C,\Q_\ell)^{I_p} \right)$.
We say that~$C$ is automorphic if~$\Lambda(C,s) = \Lambda(\pi,s)$, where~$\pi$ is an automorphic
form for~$\GL_{2g}(\Q)$, and~$\Lambda(\pi,s)$ is the
completed~$L$-function
associated to the standard representation
of~$\GL_{2g}$. If~$C$ is automorphic, then
$$\Lambda(C,s) = \pm N^{1-s} \Lambda(C,2-s),$$
where~$N$ is the conductor of~$A$.
One conjectures that all smooth projective curves~$C$ over~$\Q$ are automorphic.
 When~$g= 0$ and~$g= 1$, one knows that~$C$ is automorphic by
theorems of Riemann~\cite{Riemann} and Wiles et al.~\cite{MR1333035,MR1333036,MR1839918} respectively.
The conjecture seems completely hopeless with current technology for general
curves when~$g \ge 3$, but for~$g= 2$ it was recently proved in~\cite{BCGP} that all
such curves over~$\Q$ (and even over totally real fields) were potentially automorphic.
For abelian surfaces over~$\Q$, let us additionally  say that~$A = \Jac(C)$ is \emph{modular} of level~$N$ if there exists
a cuspidal Siegel modular form~$f$ of weight two  such that~$\Lambda(C,s) = \Lambda(f,s)$,
where~$\Lambda(f,s)$ is the completed~$L$-function
associated to the degree four spin representation of~$\GSp_4$. If~$A$ is modular in this sense, then
it is also automorphic in the sense above by taking~$\pi$ to be the transfer of the automorphic
representation associated to~$f$ from~$\GSp(4)/\Q$ to~$\GL(4)/\Q$.
It was also shown in~\cite{BCGP} that certain classes of abelian surfaces over~$\Q$ were actually modular (see Theorem~\ref{theorem:infinite} below), and
even that there were infinitely many modular abelian surfaces over~$\Q$ up to twist with~$\End_{\C}(A) = \Z$. However, no explicit examples
of such surfaces were given in that paper.

The aim of this note is to give explicit examples of modular abelian
surfaces~$A/\Q$ with~$\End_{\C}(A) = \Z$ and such that~$A$ has good
reduction outside a set~$S$ that is either~$S = \{2,5\}$, $S = \{2,5,7\}$, or~$S = \{2,3,7\}$.
Previous explicit examples of modular abelian surfaces with trivial endomorphisms were found  by~\cite{BPVY} (in 2015) and
also by~\cite{Berger}; these results
relied heavily on very delicate and explicit computations of   spaces of low weight Siegel modular forms following~\cite{MR3315514,MR3713095}.
In particular, they rely on the conductor being relatively small and also take advantage of the fact that the conductor is odd and squarefree.
(The examples in those papers are of conductors~$277$,   $353$, $587$, and~$731 = 17 \cdot 43$.)
In contrast, the examples of this paper only require verifying some local properties of~$A$
at the prime~$p$ (with~$p = 3$ or~$p = 5$)  and showing that the image of the action of~$G_{\Q}$
on the~$p$-torsion of~$A = \Jac(C)$ is of a suitable form. Although the conductors of
our examples have only small factors, the conductors themselves are quite large --- the smallest of our examples has conductor~$98000 = 2^4 \cdot 5^3 \cdot 7^2$.
The modularity of the examples in this paper follows
by applying the following result (with either~$p = 3$ or~$p = 5$) proved
in~\cite[Propositions~10.1.1 and~10.1.3]{BCGP}.

\begin{theorem} \label{theorem:infinite} Let~$A/\Q$ be an abelian
  surface with good ordinary reduction at~$v|p$ and a polarization of degree prime to~$p$,
and suppose that the eigenvalues of Frobenius on~$A[p](\Fbar_p)$ are distinct.
Let
$$\rhobar_{A,p}: G_F \rightarrow \GSp_4(\F_p)$$
denote the mod-$p$ Galois representation associated to~$A[p]$, and assume
that~$\rhobar_{A,p}$ has vast and tidy image in the notation of~\cite{BCGP}.
Suppose that either:
\begin{enumerate}
\item $p = 3$, and~$\rhobar_{A,3}$ is induced from a~$2$-dimensional
  representation over a real quadratic extension~$F/\Q$ in which ~$3$ is unramified.
\item $p = 5$, and~$\rhobar_{A,5}$ is induced from a~$2$-dimensional
  representation valued in~$\GL_2(\F_5)$  over a real quadratic
extension~$F/\Q$ in which ~$5$ is unramified.
\end{enumerate}
Then~$A$ is modular.
\end{theorem}

A precise definition of what representations are vast and tidy is included in~\S7.5 of~\cite{BCGP}, but we
content ourselves with the following list which exhausts all of our examples:

\begin{lemma}[Examples of vast and tidy representations from~{\cite[Lemmas~7.5.13 and~7.5.21]{BCGP}}]
\emph{The representation~$\rhobar_{A,p}$
is automatically vast and tidy when the image of~$\rhobar_{A,p}$ is one of the following conjugacy
classes of subgroups of~$\GSp_4(\F_p)$:
\begin{enumerate}
\item The  groups~$G_{2304}$, $G_{768}$, $G'_{768}$ or~$G_{480}$ in~$\GSp_4(\F_3)$
of orders~$2304$, $768$, $768$, and~$480$, where:
\begin{enumerate}
\item The group~$G_{2304}$ is a semi-direct product~$\Delta \rtimes \Z/2\Z$
where
$$\Delta = \left\{(A,B)\in\GL_2(\F_3)^2\mid \det(A)=\det(B)\right\};$$
it is (up to conjugacy) the unique subgroup of order~$2304$ of~$\GSp_4(\F_3)$.
\item The groups~$G_{768}$ and~$G'_{768}$ are subgroups of~$G_{2304}$ of index~$3$,
  and are (up to conjugacy) the only two subgroups of order~$768$ of~$\GSp_4(\F_3)$.
They are  isomorphic as abstract groups, but they are distinguished up to conjugacy inside~$\GSp_4(\F_3)$
by their intersections~$H_{384}$ and~$H'_{384}$ with~$\Sp_4(\F_3)$. In particular,
$(H_{384})^{{\ab}} \simeq \Z/6\Z$ and~$(H'_{384})^{{\ab}} \simeq \Z/2\Z$.
According to the small groups database of~\texttt{magma}  (cf.~\cite{groups}),
$$G_{768} \simeq G'_{768} \simeq \text{\texttt{SmallGroup(768,1086054)}},$$
whereas
$$H_{384} \simeq \text{\texttt{SmallGroup(384, 18130)}}, \quad H'_{384} \simeq \text{\texttt{SmallGroup(384, 618)}}.$$
These groups can also be distinguished by their images~$P_{192}$ and~$P'_{192}$ in~$\PSp_4(\F_3) \subset \PGSp_4(\F_3)$,
namely
$$P_{192} \simeq \text{\texttt{SmallGroup(192, 1493)}}, \quad P'_{192} \simeq \text{\texttt{SmallGroup(192, 201)}}.$$
\item The group~$G_{480}$ is a semi-direct product
$ \widetilde{A}_5 \rtimes \langle \sigma \rangle$ where~$\widetilde{A}_5 \subset \GL_2(\F_9)$ is a central
extension of~$A_5$ by~$\Z/4\Z$.
There are precisely two subgroups of this order up to conjugacy in~$\GSp_4(\F_3)$.
The second subgroup~$G'_{480}$
 also contains~$ \widetilde{A}_5$ with index two, but it is not a semi-direct
product. According to the small groups database of~\texttt{magma},
$$G_{480} \simeq  \text{\texttt{SmallGroup(480, 948)}}, \quad
G'_{480} \simeq  \text{\texttt{SmallGroup(480, 947)}}.$$
\end{enumerate}
\item The group~$G_{115200}$ in~$\GSp_4(\F_5)$ is a semi-direct product~$\Delta \rtimes \Z/2\Z$
where
$$\Delta = \left\{(A,B)\in\GL_2(\F_5)^2\mid \det(A)=\det(B)\right\};$$
it is (up to conjugacy) the unique subgroup of order~$115200$ of~$\GSp_4(\F_5)$.
\end{enumerate}
}
\end{lemma}

The conditions  of the theorem are all very easy to verify in any given example (once found) with the possible
exception of computing the image of the mod-$p$ representation for~$p = 3$ or~$5$.
 We describe how we computed this
in the section below. The second problem is then to find a list of candidate curves.
Our original approach involved searching for curves in a large box, which did indeed result
in a number of examples. However, we then switched to using a collection of curves
provided to us by Andrew Sutherland, all of which had the property that they had good reduction
outside the set~$\{2,3,5,7\}$ (these were found during the construction
 of~\cite{database} but discarded because their minimal discriminants were too large).
  This list  consisted of some~$20$ million curves,
so the next task was to identify examples to which we could apply
Theorem~\ref{theorem:infinite}.
 For a genus two curve~$C$ on Sutherland's list,
we applied the following algorithm.
\begin{enumerate}
\item Fix a real quadratic field~$F$  of fundamental discriminant~$D$ dividing~$\Delta_C$ in
which~$p \in \{3,5\}$ is unramified. Since~$\Delta_C$ is only divisible by primes in~$\{2,3,5,7\}$,
there are at most seven such~$F$. Let~$\chi_D$ denote the quadratic character associated to~$F$.
\item Check whether~$a_q \equiv 0 \mod p$  for all primes~$q \le 100$ of good reduction for~$C$
with~$\chi_D(q) = -1$.
\item Check that~$a_q \ne 0$ for at least one prime~$q \le 100$ of good reduction for~$C$
with~$\chi_D(q) = -1$.
\end{enumerate}
Any~$C$ that passes this test is likely to have the following two properties: $\rhobar_{A,p}$ is induced
from~$F$, but the~$p$-adic representation~$\rho_{A,p}$  itself is not induced. The third condition in particular guarantees that~$A$ itself is not isogenous to
a base change of an elliptic curve defined over~$F$. Note that this test is very
fast --- one can discard a~$C$ as soon as one finds a prime~$q$ with~$\chi_D(q) = -1$ and~$a_q \not\equiv 0 \mod p$, so for almost all curves~$C$, one only has to compute~$a_q$ for very small primes~$q$.
In addition, the following postage stamp calculation with the Chebotarev density theorem suggests that false
positives will be  few in number: for each of the  allowable discriminants~$D$ (there are~$7$ such~$D$
for either~$p = 3$ or~$p = 5$), there are at least~$M \ge 10$ primes in the interval~$[10,100]$  with~$\chi_D(q) = -1$. A ``random'' abelian surface~$A$ will have~$a_q \equiv 0 \mod p$ for any such prime~$q$
approximately~$1/p$ of the time (the exact expectation depends on~$A[p]$ --- if the mod-$p$
representation is surjective, the exact expectation that~$a_q \equiv 0 \mod p$ for a random prime~$q$ is~$231/640$ for~$p = 3$
and~$3095/14976$ for~$p = 5$), and so one might
expect a false positive to occur with probability approximately~$1/p^{M}$.
On the other hand, false positives are certainly not impossible.
In our original box search,
we did find the one curve~$C: y^2 = x^5 - 2x^4 + 6x^3 + 5x^2 + 10x + 5$ that ``passed'' the test for~$\rhobar_{A,3}$ to be induced from~$\Q(\sqrt{7})$,
whereas it turns out instead to be induced from~$\Q(\sqrt{85})$ --- requiring only an accidental vanishing of~$a_q$
for~$q=23$, $73$, $89$, and~$97$. The smallest prime guaranteeing that~$\rhobar_{A,3}$ is not induced from~$\Q(\sqrt{7})$  in this case
is~$a_{151} = 5 \not\equiv 0 \mod 3$.

\section{Determining the mod-\texorpdfstring{$p$}{p} representation}

Consider a genus two curve
$$C: Y^2 = f(X),$$
with~$\deg(f) = 6$. The desingularization of the corresponding projective curve has two points~$\mathfrak{b}_1$ and~$\mathfrak{b}_2$ at infinity. The canonical class~$\mathfrak{O}$ in~$\Pic^2(C)$ is represented by the divisor~$\mathfrak{b}_1 + \mathfrak{b}_2$, and the Jacobian~$A = \Jac(C)$ can be identified with~$\Pic^2(C)$ under addition of the canonical class.
By Riemann--Roch, every class in~$\Pic^2(C)$ except $\mathfrak{O}$ has precisely one effective divisor. Thus, we may represent any point of~$A$ as an unordered pair~$\{P,Q\}$ of points on~$C$.

If we assume~$f(X)$ has a rational root, then, by suitably transforming the variables~$X$ and~$Y$, we can make~$\deg(f) = 5$; then, there will be exactly one point~$\mathfrak{b}$ at infinity, and the canonical class will be represented by~$2\mathfrak{b}$. We will not need this assumption, however, and several of our examples do not have
any Weierstrass points over~$\Q$.

\subsection{\texorpdfstring{$p = 3$}{p=3}}
Let~$K/\Q$ denote the Galois closure of the corresponding projective representation. It will contain the field~$\Q(x+u,xu,yv)$ for any~$3$-torsion point $\{P,Q\}$ of $A$, where~$P = (x,y)$ and~$Q = (u,v)$.
There exist polynomials $B_{ij}$, given in \cite[Theorem~3.4.1 and Appendix \uppercase\expandafter{\romannumeral2\relax}]{cassels_flynn_1996}, using which the multiplication-by-$n$ map can be described explicitly at the level of the Kummer surface of $A$. Writing the equation $[2]\{P,Q\} = -\{P,Q\}$ in terms of the Kummer coordinates explicitly, taking resultants, and eliminating spurious solutions, one can compute the minimal
polynomials of~$x+u$,~$xu$ and~$yv$ in any particular case, as well as determine
the Galois group of the corresponding extension.

Note that the first coordinates determine
the~$\GSp_4(\F_3)/\langle \pm 1 \rangle = \PGSp_4(\F_3)$-representation, so this determines the image of~$\rhobar_{A,3}$
modulo the central subgroup of order~$2$ as an abstract group.
One can similarly compute the field~$\Q(y+v,yv)$ if one wants to know the full~$\GSp_4(\F_3)$-representation.
In any case of interest, this is enough (purely
by considering possible orders) to determine the order of the image of~$\rhobar_{A,3}$ itself.
It then remains to determine the precise subgroup of~$\GSp_4(\F_3)$ in the cases where
this is ambiguous.
The group~$\PGSp_4(\F_3)$ has a natural permutation representation on~$40$ points, corresponding
to the non-zero points of~$A[3]$ up to sign (warning: the group~$\PGSp_4(\F_3)$ has a second non-conjugate
representation on~$40$ points). From this data, one can distinguish between~$G_{480}$
and~$G'_{480}$ purely based on the degrees of the polynomials arising from the computation above. The following table
gives the corresponding decomposition in the cases of interest:
\begin{center}
\begin{tabular}{|l|l|}
\hline
$G$ & Orbits \\
\hline
$G_{2304}$ & $8$, $32$ \\
$G_{768}$ &  $8$, $32$  \\
$G'_{768}$ &  $8$, $32$  \\
$G_{480}$ & $20$, $20$ \\
$G'_{480}$ & $40$ \\
\hline
\end{tabular}
\end{center}
The groups~$G_{768}$ and~$G'_{768}$ cannot be distinguished by this method. This is not
important for establishing modularity since both groups give representations with vast
and tidy image. However, in order to complete the tables, we distinguish between
these cases as follows: we \emph{explicitly} compute (using \texttt{magma})
the Galois group of the corresponding degree~$32$ polynomial over the field~$\Q(\sqrt{-3})$,
and see whether the resulting group is~$P_{192}$ or~$P'_{192}$ (in which case the
group is~$G_{768}$ or~$G'_{768}$ respectively).

\subsection{\texorpdfstring{$p = 5$}{p=5}}
Similar to the $p=3$ case, for an arbitrary point $\{P = (x,y),Q = (u,v)\}$ of $A$, we write the equation $3\{P,Q\} = -2\{P,Q\}$ in terms of the Kummer coordinates of the point, and take resultants to find the minimal polynomials of $x+u, xu$ and $yv$ of $5$-torsion points on $A$. The splitting field of these polynomials is the Galois closure~$K/\Q$ of the representation to $\PGSp_4(\F_5)=\GSp_4(\F_5)/\langle \pm 1 \rangle$.

We describe an algorithm for showing that the image~$\rhobar_{A,5}$
of a mod-$5$ representation in~$\GSp_4(\F_5)$ with cyclotomic determinant has image~$G_{115200}$.
The group~$\GSp_4(\F_5)$ has a representation on~$312 = (5^4 - 1)/2$ points, given
by the action on the non-trivial~$5$-torsion points up to sign (which factors
through~$\PGSp_4(\F_5)$).

\begin{lemma} Let~$G \subset \GSp_4(\F_5)$ be a subgroup, and suppose that
the similitude character is surjective on~$G$, or equivalently that~$[G: G \cap \Sp_4(\F_5)] = 4$.
Suppose, in addition, that~$G$ acts on the degree~$312$ permutation representation
above with two orbits of size~$288$ and~$24$ respectively.
Then:
\begin{enumerate}
\item $G$ is one of four groups, distinguished by their orders: $2304$, $4608$, $57600$, and~$115200$.
\item The degree~$24$ permutation representation of~$G$ factors through a group of
order~$576$, $1152$, $14400$, and~$28800$ respectively.
\end{enumerate}
\end{lemma}

In particular, we can distinguish these representations by computing the Galois group
of the factor of size~$24$.
Hence by computing the corresponding polynomials of order~$24$ and~$288$ we
can verify that the image is indeed~$G_{115200}$.

\subsection{Checking the Sato--Tate group}

For all the residual representations we consider, it turns out that the image of~$\rhobar$
is big enough to guarantee that the Sato--Tate group is either~$\USp(4)$ or the normalizer of~$\SU(2) \times \SU(2)$.
More precisely:

\begin{lemma} Suppose that~$p = 3$ and that~$\rhobar_{A,p}$ has image
either~$G_{480}$, $G_{768}$, $G'_{768}$, $G_{2304}$, or that~$p = 5$,
and~$\rhobar_{A,p}$ has image~$G_{115200}$. Then the Sato--Tate group of~$A$
is either~$\USp(4)$ or~$N(\SU(2) \times \SU(2))$. Moreover,  if the Sato--Tate
group is~$N(\SU(2) \times \SU(2))$, the quadratic extension~$F/\Q$ over which~$A$ has
Sato--Tate group~$\SU(2) \times \SU(2)$ is the quadratic field~$F$ from which~$\rhobar$
is induced.
\end{lemma}

\begin{proof} The image of~$\rhobar_{A,p}$ is constrained by the Sato--Tate group,
and thus the fact that the Sato--Tate group
can only be~$\USp(4)$ or~$N(\SU(2) \times \SU(2))$ follows directly from a classification of all such groups in~\cite{MR2982436}.
(In fact, when the image is~$G_{480}$, only the first case can occur.)
In the latter case, the representation becomes reducible over the quadratic extension~$F$ where~$A$
has Sato--Tate group~$\SU(2) \times \SU(2)$, and (for the given~$\rhobar$) this forces~$F$ to be the field from which~$\rhobar$
is induced.
\end{proof}

In particular, in all our examples, our initial selection process requires the existence of a prime~$q$ of good reduction
with~$\chi(q) = -1$ and~$a_q \ne 0$, which implies that~$\rho_{A,p}$ cannot be induced from~$F$, and thus
the Sato--Tate group in each example below is~$\USp(4)$.

\section{Examples}
Of the curves we consider, a number satisfy the conditions of the main theorem,
and are thus provably modular.  For any curve~$C$
that is modular, so too are any quadratic twists.
Hence we only list a single representative curve for each equivalence class of abelian surfaces
under both~$\Q$-isogenies and twisting by quadratic characters.

\subsection{Inductions from~\texorpdfstring{$\GL_2(\F_3)$}{GL2(F3)} and~\texorpdfstring{$\GL_2(\F_9)$}{GL2(F9)}}

We first give the examples of modular curves
whose mod-$3$ representation is induced from either~$\GL_2(\F_3)$ or~$\GL_2(\F_9)$-representations of~$G_F$
for real quadratic fields~$F$. It turns out that, in the range of our computation, the
representation~$\rhobar$ up to twist determined the representation~$\rho$ up to twist
--- after applying our other desiderata, including that~$A/\Q$ had good reduction at~$p$
and had Sato--Tate group~$\USp(4)$. In particular, all the examples below
give rise to mod-$3$ representations that are not twist equivalent.
The examples~$C$ we choose to list are of minimal conductor amongst all those
with Jacobian isogenous to a twist of~$\Jac(C)$.
The conductors
were computed rigorously away from~$2$ using \texttt{magma}. The conductors at~$2$
were computed for us by Andrew Sutherland using an analytic algorithm
discussed in~\S5.2 of~\cite{database}. This computation \emph{assumes} the analytic
continuation and functional equation for~$L(A,s)$, which we know to be true in this case.
(More precisely, as explained to us by Andrew Booker,
one version of this program gives a non-rigorous computation of these conductors and a second
slower but more rigorous version then confirms these values.)
In the case of ties, we chose
the curve with smaller  minimal discriminant. In the case of subsequent ties, we eyeballed the different
forms and chose the one that looked the prettiest.

\begin{theorem} The Jacobians~$A=\Jac(C)$ of the following smooth genus two curves~$C$ over~$\Z[1/70]$ are modular.
In particular, the~$L$-function~$L(A,s)$ is holomorphic in~$\C$ and satisfies the corresponding
functional equation.
Each $A$ has good ordinary reduction at~$3$ and is~$3$-distinguished and~$\End_{\C}(A) = \Z$.
Moreover, the representation~$\rhobar_{A,3}$
is induced from a~$\GL_2(\Fbar_3)$-valued representation  of~$G_F$ that is vast and tidy.
{\footnotesize
\begin{center}
\begin{tabular}{|l|l|l|l|l|}
\hline
Curve & Cond & Disc & $\im(\rhobar)$ & $\Delta_F$ \\
\hline
$y^2 = x^6-10 x^4+2 x^3+31 x^2-13 x-18$ & $2^45^37^2$ & $2^85^37^3$ &  $G_{480}$ & $5$ \\
$y^2 = -5 x^6-20 x^5-10 x^4+36 x^3+22 x^2-20 x$ & $2^{10}5^37$ & $2^{20}5^47^3$ & $G'_{768}$ & $5$ \\
$y^2 + y = -4 x^5-23 x^4-22 x^3+74 x^2-40 x+6$ &  $2^85^37^2$ & $2^{19}5^77^2$ & $G_{2304}$ & $5$ \\
$y^2 = 16 x^6-46 x^4+10 x^3+46 x^2-9 x-17$ &  $2^{12}5^27^4$ & $2^{19}5^97^4$ & $G_{480}$ & $5$ \\
 \hline
 $y^2 =2 x^5-8 x^4+26 x^2-7 x-26$ & $2^{15}5$ & $2^{16}5^3$  & $G_{2304}$ & $8$ \\
 $y^2 = x^5-x^4-4 x^3-44 x^2-60 x-100$ &  $2^{14}5\cdot 7$ & $2^{33}5^37$ & $G_{2304}$ & $8$ \\
 $y^2 =x^5-17 x^4+70 x^3+26 x^2-35 x-29$ & $2^{16}5\cdot 7$ &  $2^{37}5^37$ &  $G_{2304}$ & $8$ \\
  $y^2 + x^2 y = 13 x^6-29 x^5-10 x^4+41 x^3+6 x^2+20 x+20$ & $2^75^27^4$ & $2^{16}5^27^{16}$
  & $G_{2304}$ & $8$ \\
  $y^2 = x^5-11 x^4-2 x^3-34 x^2-5 x-25$ & $2^{20}5\cdot 7$ & $2^{21}5^37^3$ & $G_{768}$ & $8$ \\
  $y^2 = -2 x^6-41 x^5-48 x^4+54 x^3+42 x^2-49 x$ & $2^{14}5^27^4$ & $2^{32}5^27^{11}$ &
 $G_{2304}$ & $8$ \\
 $y^2 = 2 x^5+34 x^4-16 x^3-52 x^2-13 x-1$ & $2^{19}5^37^2$ & $2^{20}5^57^6$ &  $G'_{768}$ & $8$ \\
 $y^2 = 8 x^6-24 x^5-4 x^4+20 x^3+49 x^2-21 x-28$ & $2^{15}5^27^4$ & $2^{23}5^67^9$ & $G_{2304}$ & $8$ \\
\hline
$y^2 + (x+1)y = 64 x^5-8 x^4+39 x^3+x^2+2 x+1$ & $2^75^37^3$ &  $2^{27}5^67^6$ & $G_{480}$ & $40$ \\
$y^2 = 15 x^5+23 x^4+20 x^3+28 x^2+12 x-4$ & $2^{14}3 \cdot 5^3$ & $2^{33}3^25^4$ &  $G_{2304}$ & $40$ \\
$y^2 = 3 x^5+7 x^4+28 x^3+20 x^2+28 x-36$ &  $2^{14}3 \cdot 5^3$ & $2^{36}3^25^4$ &  $G_{2304}$ & $40$ \\
\hline
\end{tabular}
\end{center}
}
\end{theorem}

\begin{example} Precisely one curve in this table is actually smooth over a smaller ring, namely
  the curve of conductor~$163840=2^{15}\cdot 5$ which is smooth over~$\Z[1/10]$. This curve has a quadratic
twist with particularly small na\"{i}ve height, namely the curve:
$$y^2 = 4 x^5 + 6 x^4 + 4 x^3 + 6 x^2 + 2 x + 3$$
which also has conductor~$163840 = 2^{15} \cdot 5$ but larger minimal
discriminant
$$131072000000 = 2^{23} \cdot 5^6$$
rather than~$8192000 = 2^{16} \cdot 5^3$ as the curve
in the table.
The mod-$3$ representation of both of these curves is actually unramified at~$5$, and is congruent
up to twist to the mod-$3$ representation attached to the curve
$y^2 = 4 x^5 - 4 x^4 + 4 x^3 - 2 x^2 + x$
of conductor~$2^{15}$. The Jacobian of this latter curve is isogenous to~$\Res_{\Q(\sqrt{2})/\Q}(E)$, where~$E$ is the elliptic curve:
$$y^2  + \sqrt{2} x y = x^3 + (-1 - \sqrt{2}) x^2 + 2(\sqrt{2}+1)x - 3 \sqrt{2} - 5.$$
\end{example}

\subsection{Inductions from~\texorpdfstring{$\GL_2(\F_{5})$}{GL2(5)}}

We now consider the case~$p = 5$.

\begin{theorem}  \label{theorem:miraclecurve}
  The Jacobians~$A=\Jac(C)$ of the following smooth genus two curves~$C$ over~$\Z[1/42]$ are modular.
In particular, the~$L$-function~$L(A,s)$ is holomorphic in~$\C$ and satisfies the corresponding
functional equation.
Each~$A$ has good ordinary
reduction at~$5$ and is~$5$-distinguished and~$\End_{\C}(A)=\Z$. Moreover, the representation~$\rhobar_{A,5}$
is induced from a~$\GL_2(\F_5)$-valued representation  of~$G_F$ that is vast and tidy.
{\footnotesize
\begin{center}
\begin{tabular}{|l|l|l|l|l|}
\hline
Curve & Cond & Disc & $\im(\rhobar)$ & $\Delta_F$ \\
\hline
$y^2  + xy = 7 x^6-22 x^5-7 x^4+61 x^3-3 x^2-54 x-12$ & $2^73^27^3$ & $2^{11}3^97^4$
& $G_{115200}$ & $8$ \\
$y^2 = 8 x^6-24 x^5-30 x^4+8 x^3-24 x^2-48 x-8$
& $2^63^87$ & $2^{51}3^87$ & $G_{115200}$  & $8$ \\
\hline
\end{tabular}
\end{center}
}
\end{theorem}

The second curve also admits a quadratic twist of smaller na\"{i}ve height, namely
$$y^3 + x^2 y = x^6 - 3 x^5 - 4 x^4 + x^3 - 3 x^2 - 6 x - 1$$
of conductor~$5878656 = 2^7 \cdot 3^8 \cdot 7$ and minimal discriminant~$96315899904 = 2^{21} \cdot 3^8
\cdot 7$.

\section{Acknowledgments}
We would like to thank Andrew Booker, Andrew Sutherland, and John Voight for useful discussions about this project.
We would also like to thank Andrew Sutherland for providing us with a large list of genus two
curves over~$\Q$ with good reduction outside~$\{2,3,5,7\}$. Finally, we would  like to thank Andrew Sutherland and Andrew Booker for help computing the~$2$-part of the conductors of our curves.

\bibliographystyle{amsalpha}
\bibliography{Examples}

\end{document}